\documentclass[12pt]{article}

\usepackage{amsmath,amssymb,amsthm,graphicx}
\usepackage{a4wide}

\newtheorem{lemma}{Lemma}
\newtheorem{prop}[lemma]{Proposition}

\numberwithin{lemma}{section}
\numberwithin{theorem}{section}
\numberwithin{fact}{section}
\numberwithin{equation}{section}

\title{Volume Laws for Boxed Plane Partitions and Area Laws for Ferrers Diagrams}

\author{U.~Schwerdtfeger\\
\\
Fakult\"at f\"ur Mathematik\\
Universit\"at Bielefeld\\
Postfach 10 01 31, 33501 Bielefeld, Germany\\
}

\begin{document}

\maketitle

\begin{abstract}We asymptotically analyse the volume-random variables of general, symmetric and cyclically symmetric plane partitions fitting inside a box. We consider the respective symmetry class equipped with the uniform distribution. We also prove area limit laws for two ensembles of Ferrers diagrams. Most of the limit laws are Gaussian.
\end{abstract}

\section{Introduction}
A \emph{plane partition fitting inside an $(r,s,t)$-box} is an $r \times s$-array $\Pi$ of non-negative integers $p_{i,j} \le t$ with weakly decreasing rows and columns. It can be visualised as a pile of unit cubes in the box  ${\cal B} (r,s,t):= \left[0,r \right ]\times \left[0,s \right ]\times \left[0,t \right ] $ ``flushed into the corner", see figure \ref{pileofcubes} and \cite{Br}. The figure also illustrates the connection to tilings of a hexagon of side lengths $r,s$ and $t$ by lozenges. Yet another interpretation is viewing such a pile of cubes as an order ideal in the product of three finite chains (total orders) with the respective lengths $r,s$ and $t.$ In the following we mean by ``plane partition" one which fits inside an $(r,s,t)$-box. The volume of a plane partition is the sum of its parts or the number of unit cubes in the pile or the cardinality of the order ideal, respectively. If one side length, say $t$ of the bounding box is one, such a pile can be regarded as Ferrers diagram fitting inside an $r\times s$-rectangle of area equal to the volume. Let ${\cal PP}(r,s,t)$ denote the set of all plane partitions fitting inside ${\cal B}(r,s,t).$ A plane partition $\Pi\in{\cal PP}(r,r,t)$ is called \emph{symmetric} if the corresponding pile of cubes is symmetric about $x=y.$ To put it differently the cube $(i,j,k)$ belongs to the pile, if and only if $(j,i,k)$ does. The respective tiling is symmetric w.r.t. a vertical axis.  A plane partition $\Pi\in{\cal PP}(r,r,r)$ is called \emph{cyclically symmetric} if the corresponding tiling is invariant under a rotation about $2\pi/3,$ i.e. the cube $(i,j,k)$ belongs to the pile if and only if $(k,i,j)$ and $(j,k,i)$ do. Call these subsets ${\cal SPP}(r,t)$ and ${\cal CSPP}(r)$ respectively. Several authors discussed properties of randomly chosen plane partitions, for example in \cite{Wi} diagonal sums are considered. In \cite{CLP} an ``Arctic Circle Theorem" for tilings of a large hexagon is proved, which states that a typical tiling looks periodic near the corners of the hexagon and unordered inside the inscribed ellipse. \\
It is natural to ask also for volume limit laws as the generating functions counting plane partitions by volume are available for the above classes. For the following formulae we refer to \cite{Br}. The volume generating function of ${\cal PP}(r,s,t)$ is
\begin{equation} \label{gfpp}
\prod_{i=1}^r\prod_{j=1}^s\prod_{k=1}^t \frac{1-q^{i+j+k-1}}{1-q^{i+j+k-2}}.
\end{equation}

\begin{figure}\label{pileofcubes}
\begin{center}
\includegraphics[height=70mm,width=80mm]{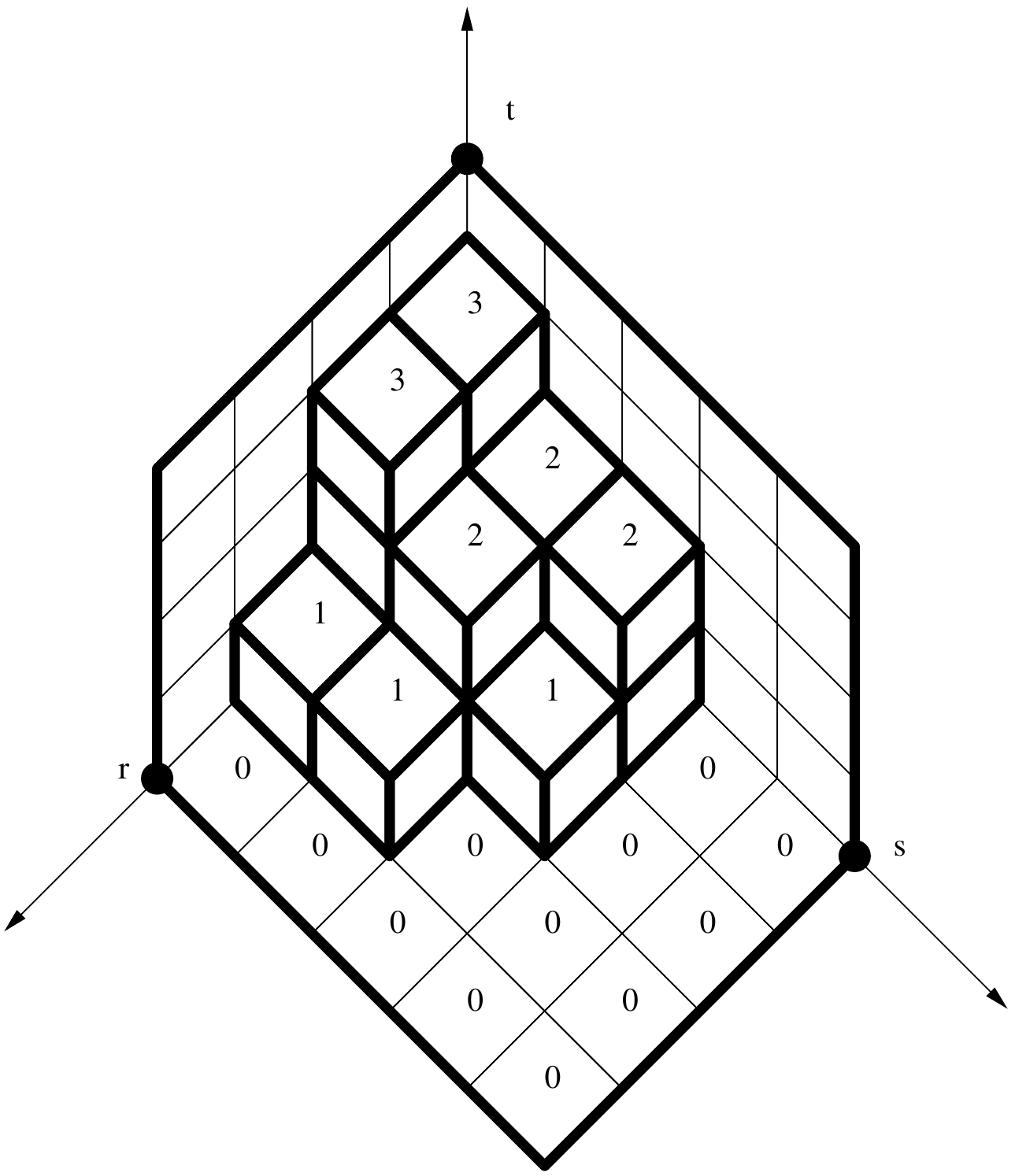}
\caption{A plane partition in $(4,5,4)$-box}
\end{center}
\end{figure}
Symmetric plane partitions in ${\cal B}(r,r,t)$ have the volume generating function
\begin{equation}\label{sympp}
\left(\prod_{i=1}^r\prod_{k=1}^t\frac{1-q^{2i+k-1}}{1-q^{2i+k-2}}\right)
\left( \prod_{1\le i<j\le r}\prod_{k=1}^t   \frac{1-q^{2+2(i+j+k-2)}}{1-q^{2(i+j+k-2)}}    \right).\end{equation}
Here the first product corresponds to singleton orbits and the second product corresponds to doubleton orbits of ${\cal B}(r,r,t)$ under switching the first two coordinates.
For the cyclically symmetric plane partitions we have the volume generating function
\begin{equation} \label{cspp}
\left(\prod_{1\le i<j<k\le r} \frac{1-q^{3(i+j+k-1)}}{1-q^{3(i+j+k-2)}}\right)^2
\left(\prod_{1\le i<k\le r} \frac{1-q^{3(2i+k-1)}}{1-q^{3(2i+k-2)}}\cdot\frac{1-q^{3(i+2k-1)}}{1-q^{3(i+2k-2)}}\right) \left(\prod_{i=1}^r \frac{1-q^{3i-1}}{1-q^{3i-2}}\right), 
\end{equation}
where the first product runs over ${\cal C}_3$ orbits of ${\cal B}(r,r,r)$ with all coordinates distinct, the second runs over orbits with exactly two coordinates equal and the third corresponds to singleton orbits. The above generating functions are polynomials in $q$ whose degree is the volume of the bounding box.\\
We equip ${\cal PP}(r,s,t),$ ${\cal SPP}(r,t)$ and ${\cal CSPP}(r)$ with the uniform distribution (``uniform fixed bounding box ensemble"). By $X=X_{rst}$ (resp. $=X_{rt},\;X_r$) denote the corresponding volume random variables. In the subsequent discussion we will drop subscripts as the side lengths are always $r,s$ and $t$ (resp. $r$ and $t$, $r$).  The probability generating function of this random variable is $P(q):=G(q)/G(1)$ where $G(q)$ is the respective generating function. In the ${\cal PP}(r,s,t)$ case this reads with the notion $\alpha_{ijk}:=i+j+k-1$
\begin{equation}\label{PGF}
P(q)=\prod_{i=1}^r\prod_{j=1}^s\prod_{k=1}^t \frac{(\alpha_{ijk}-1)\left(1-q^{\alpha_{ijk}}\right)}{\alpha_{ijk}\left(1-q^{\alpha_{ijk}-1}\right)}.
\end{equation}
 In the following $g(x)^{(N)}$ denotes the $N^{th}$ derivative of $g$ w.r.t. $x.$ Moments of $X$ of arbitrary order $N$ exist and can be computed via 
\begin{equation}
\mathbb{E}\left(X^N \right ) =(-I)^N\left .\left( P(e^{Ix})\right)^{(N)} \right |_{x=0},
\end{equation}
where $P(e^{Ix})$ ($I^2=-1$) is the characteristic function of $X.$
\section{Mean, variance and concentration properties}
Denote by $\mu,$ $\mu_{spp}$ and $\mu_{cspp}$ the mean value and by $\sigma,$ $\sigma_{spp}$ and $\sigma_{cspp}$ the standard deviation of the volume variables on ${\cal PP}(r,s,t),$ ${\cal SPP}(r,t)$ and ${\cal CSPP}(r),$ respectively.
\begin{lemma}\label{expectation}
The volume distributions are symmetric about half the volume of the bounding box, and hence the expected volumes are $\mu=rst/2,$ $\mu_{spp}=r^2t/2$ and $\mu_{cspp}=r^3/2.$ 
\end{lemma}
\begin{proof}
Let $v$ be the volume of the bounding box ${\cal B}.$ We show that $X$ and $v-X$ are equally distributed. Let a plane partition of volume $k$ be represented by a pile of $k$ green cubes inside ${\cal B}.$ Now fill up the empty space in  ${\cal B}$ with red cubes. The red cubes represent a plane partition of volume $v-k.$ This construction describes a bijection between plane partitions of volume $k$ and $v-k$ which respects symmetry and cyclic symmetry. Now the lemma follows as $\mathbb{E}(X)=\mathbb{E}(v-X).$
\end{proof}
In order to prove the Gaussian limits we consider the characteristic functions $P(e^ {Ix})$ of the random variables $X.$ More precisely, since sums are easier to handle than products, we will study the logarithms of the characteristic functions. The following lemma  enables us to compute explicit formulas for the variances of the random variables and to estimate the Taylor coefficients of the functions in question.
\begin{lemma}\label{sumest}
For positive real numbers $\alpha,c$ with $\alpha>c\ge 1,$ we have the expansion
\begin{equation}\label{Abl}
\log\left( \frac{(\alpha-c)\left(1-e^{\alpha x}\right)}{\alpha\left(1-e^{(\alpha-c) x}\right)}\right)
=\sum_{N \ge 1} H_{N,c}(\alpha)x^N.
\end{equation}
Here $H_{N,c}$ is a polynomial of degree $N-1$ in $\alpha.$ In particular we have $H_{1,c}(\alpha)=\frac{1}{2}$ and $H_{2,c}(\alpha)=-\frac{c}{12}\alpha+\frac{c^2}{24}.$ Furthermore there is a positive constant $D,$ such that the inequality
\begin{equation}\label{TaylorEst}
\left |H_{N,c}(\alpha)\right|\le D\cdot \alpha^{N-1}\cdot (2c)^N
\end{equation}
holds for all $N\in \mathbb{N}.$ 
\end{lemma}
\begin{proof}
Define the function $g(t)$ by
\begin{equation}\label{defg}
g(t)= \log \left(\frac{e^{t}-1}{t}\right).
\end{equation}
Then the lhs of \eqref{Abl} is easily seen to be equal to $g\left(\alpha x\right)-g\left((\alpha-c)x\right).$
We have the following series expansion for $g(t):$
\begin{equation}
g(t)= \log \left(\frac{e^{t}-1}{t}\right)=\log \left(1+\frac{1}{2}t+\frac{1}{3!}t^2+\ldots\right)=\sum_{N\ge1}b_Nt^N.
\end{equation}
Observe that the singularities of $g$ of smallest modulus are $\pm 2\pi I,$ so the numbers $\left|b_N\right|$ decay like $(2\pi)^{-N}$ and hence are bounded by some constant $D.$ The  $N^{th}$ coefficient in the Taylor expansion of $g\left(\alpha x\right)-g\left((\alpha-c)x\right)$ about $x=0$ is in fact a polynomial of degree $N-1,$ namely
\begin{equation}
H_{N,c}\left(\alpha\right)=b_N\cdot\left(\alpha^N-(\alpha-c)^N \right)=b_N\cdot \sum_{k=1}^N \binom{N}{k}(-1)^{k+1}c^k\alpha^{N-k}.
\end{equation}
As $\alpha>c\ge 1$ it can be estimated as follows:
\begin{equation}
\left|b_N\cdot \sum_{k=1}^N (-1)^{k+1}\binom{N}{k}c^k\alpha^{N-k}\right| \le D\cdot\alpha^{N-1}\cdot c^N\cdot \sum_{k=1}^N \binom{N}{k} \le D\cdot \alpha^{N-1}\cdot c^N\cdot2^N.
\end{equation}
This finishes the proof of Lemma \ref{sumest}.
\end{proof}
Now we can easily compute the variances. The formula for $\sigma^2$ already appeared in \cite{Wi}. 
\begin{lemma}\label{var} We have 
\begin{enumerate}
\item $\sigma^2=\frac{1}{12}(r^2st+rs^2t+rst^2)=\frac{1}{12}rst(r+s+t),$
\item $\sigma^2_{spp}=\frac{1}{3}tr^3+\frac{1}{6}t^2r^2-\frac{1}{12}t^2r+\frac{1}{6}tr^2-\frac{1}{3}tr$ and
\item $\sigma^2_{cspp}=\frac{3}{4}r^4-\frac{1}{2}r^2.$ 
\end{enumerate}
\end{lemma}
\begin{proof}
Recall that the variance of a random variable $Y$ can be obtained as the $\mathbb{V}(Y)=-\left.\log \left( P(e^{Ix})\right)^{\prime\prime} \right|_{x=0},$ where $P(q)$ is the probability generating function of $Y.$ In the ${\cal PP}(r,s,t)$ case we apply this to $P(q)$ as in \eqref{PGF} and obtain 
\begin{equation}\label{logPGF}
-\left.\log \left( P(e^{Ix})\right)^{\prime\prime}\right|_{x=0}=-\sum_{i=1}^r\sum_{j=1}^s\sum_{k=1}^t \left.\log\left(\frac{(\alpha_{ijk}-1)\left(1-e^{\alpha_{ijk}Ix}\right)}{\alpha_{ijk}\left(1-e^{(\alpha_{ijk}-1)Ix}\right)}\right)^{\prime\prime}\right|_{x=0},
\end{equation}
where $\alpha_{ijk}=i+j+k-1.$ According to Lemma \ref{sumest}, each summand on the rhs is equal to 
\begin{equation}
2!H_{2,1}(\alpha_{ijk})=-\frac{1}{6}\alpha_{ijk}+\frac{1}{12}
\end{equation}
and a straightforward calculation yields
\begin{equation}\label{varest}
\sum_{i=1}^r\sum_{j=1}^s\sum_{k=1}^t \left(\frac{1}{6} (i+j+k-1) -\frac{1}{12}\right)=\frac{1}{12}\left(r^2st+rs^2t+rst^2 \right). 
\end{equation}
The other variances are calculated in an analogous way with suitable choices of $\alpha$ and $c.$
 \end{proof}
 Now we can investigate concentration properties of the families of volume random variables, i.e. we study the quotients of standard deviation and mean when the box gets large. According to Lemma \ref{expectation} and Lemma \ref{var} these quotients tend to zero if at least two side lengths of the bounding box tend to infinity. The latter is in particular satisfied when $r\to \infty$ in the ${\cal SPP}(r,t)$ and ${\cal CSPP}(r)$ case. In the general case of ${\cal PP}(r,s,t)$ we have 
\begin{equation}\label{conpro}
\left(\frac{\sigma}{\mu}\right)^2 =\frac{1}{3} \left( \frac{1}{st}+\frac{1}{rt}+\frac{1}{rs} \right).
\end{equation}
If say $r$ and $s$ are unbounded, the right hand side of \eqref{conpro} tends to zero for $r,s \to \infty$ and the family is concentrated to the mean, i.e. for every $\delta>0$ we have
\[
 \mathbb{P}\left( 1-\delta \le \frac{X}{\mathbb{E}(X)} \le 1+\delta \right) \longrightarrow 1\qquad (r,s \to \infty) 
\]
On the other hand, if two coordinates are fixed, say $r,s,$ then the quotient is bounded away from zero by $\frac{1}{3rs}$ for $t \to \infty$. These families are not concentrated. Similar considerations hold in the case of ${\cal SPP}(r,t)$ when $r$ is fixed and $t \to \infty.$ In the next section we investigate the limit laws in these two cases.
\section{Limit laws}
We first give a result for the concentrated families.
\begin{prop}\label{concpp}
If at least two of $r,s,t$ tend to infinity, the family 
\[
\left(Y:=\frac{X-\mu}{\sigma}\right)
\]
of normalised volume random variables on ${\cal PP}(r,s,t)$ converges in distribution to a standard normal distributed random variable. The same statement holds for the volume variables on ${\cal SPP}(r,t)$ if at least $r \to \infty$ and on ${\cal CSPP}(r)$ if $r\to \infty.$
\end{prop}
\begin{proof}For ${\cal PP}(r,s,t).$ Let $r,s\to \infty$ and $t$ vary arbitrarily.  The characteristic function of $Y$ is $\phi(x):=e^{-\mu Ix/\sigma}P\left(e^{Ix/\sigma}\right),$ with $P$ as in eq. \eqref{PGF}. 
We prove that $\log\left(\phi(x)\right)\to -\frac{x^2}{2}$ for $x\in \mathbb{R}$ if $r,s\to \infty.$ Then by Levy's continuity theorem \cite{Fe} the assertion follows.\\
By Lemma \ref{expectation} $Y$ and $-Y$ have the same distribution. So the characteristic function is real valued and the coefficients for odd $N$ vanish. By Lemma \ref{sumest} we have the Taylor expansion
\begin{equation}\label{TayExp}
\phi(x)=-\frac{x^2}{2}+\sum_{N\ge 2}\sum_{i=1}^r\sum_{j=1}^s\sum_{k=1}^t 
H_{2N,1}(i+j+k-1)\left(\frac{xI}{\sigma}\right)^{2N}
\end{equation}
According to the estimate \eqref{TaylorEst} we can bound the modulus of the $2N^{th}$ summand, $N\ge 2$ by
\begin{equation}\label{coeffest1}
D \sum_{i=1}^r\sum_{j=1}^s\sum_{k=1}^t 2^{2N}(i+j+k-1)^{2N-1}\left(\frac{|x|}{\sigma}\right)^{2N}\le 4^ND |x|^{2N}\frac{rst(r+s+t)^{2N-1}}{\sigma^{2N}}.
\end{equation}
Plugging the explicit expression for $\sigma^2$ of Lemma \ref{var} into the rhs of \eqref{coeffest1} we obtain the estimate for the $2N^{th}$ summand, $N\ge 2,$ in the expansion \eqref{TayExp}   
\begin{equation}\label{coeffest2}
\left|\sum_{i=1}^r\sum_{j=1}^s\sum_{k=1}^t 
H_{2N,1}(i+j+k-1)\left(\frac{xI}{\sigma}\right)^{2N}\right|\le D |x|^{2N}\cdot48^{N}\left(\frac{1}{rs}+\frac{1}{rt}+\frac{1}{st}\right)^{N-1}.
\end{equation}
Summing in \eqref{coeffest2} over $N\ge2$ then yields the estimate 
 \begin{equation}
 \left|\log(\phi(x))+x^2/2 \right|\le 48D|x|^2\frac{48|x|^2\left(\frac{1}{rs}+\frac{1}{rt}+\frac{1}{st}\right)}{1-48|x|^2\left(\frac{1}{rs}+\frac{1}{rt}+\frac{1}{st}\right)} .
 \end{equation}
The rhs tends to zero for any fixed real $x$ if $r$ and $s$ tend to infinity. This proves the assertion for the ${\cal PP}(r,s,t)$ case. The other cases are shown with similar estimates.
\end{proof}
Now we consider the non-concentrated case. If $r=s=1$ we have a single column of unit cubes of height $t.$ The volume clearly is uniformly distributed. So, if $rs>1$ the volume random variable can be viewed as sum of dependent random variables with values in $\{0,\ldots,t\}.$ Now the easiest guess is that for large $t$ the dependence vanishes and the volume of a single column is uniformly distributed. This guess is the right one as the following proposition shows.
\begin{prop}\label{nonconcpp}
If $r$ and $s$ are fixed and $t$ tends to infinity, the family 
\[
\left(Z_t:=\frac{X_{rst}}{t}\right)
\]of rescaled random variables converges in distribution to the $(rs)$-fold convolution of the uniform distribution on $[0,1].$  In the ${\cal SPP}(r,t)$ case the so rescaled sequence converges in distribution to the convolution of $r$ factors of the uniform distribution on $[0,1]$ and $r(r-1)/2$ factors of the uniform distribution in $[0,2].$    
\end{prop}
\begin{proof}
For ${\cal PP}(r,s,t).$ We show that the Fourier transform of $Z_t$ converges pointwise to $\left(\frac{e^{Ix}-1}{Ix}\right)^{rs},$ which is the Fourier transform of the $rs$-fold convolution of the uniform distribution on $[0,1].$ The Fourier transform of $Z_t$ is $P(e^{Ix/t}),$ with $P$ as in \eqref{PGF}. We expand the product running from $1$ to $t$ in $P(e^{Ix/t}).$ All but the last term in the numerator and the first term in the denominator cancel out:
\[
P(e^{Ix/t})=\prod_{i=1}^r\prod_{j=1}^s \frac{(i+j-1)(1-e^{Ix(i+j+t-2)/t})}{(i+j+t-1)(1-e^{Ix(i+j-2)/t})}
\]
The single factors are easily seen to converge to $\frac{e^{Ix}-1}{Ix}.$\\
The ${\cal SPP}(r,t)$ case is worked out analogously.
\end{proof}
\section{Ferrers Diagrams}
A Ferrers diagram is a convex lattice polygon which contains both upper corners and the lower left corner of its smallest bounding rectangle. 
When $t=1$ a boxed plane partition can be viewed as a Ferrers diagram fitting inside an $r\times s$ rectangle. Then formula \eqref{gfpp} reduces to the $q$-binomial coefficient $\left[^{r+s}_{\;\;s} \right]_q.$ The class of Ferrers diagrams with $h$ rows and $w$ columns has the area generating polynomial 
\begin{equation}\label{hwgf}
q^{h+w-1}\left[^{h+w-2}_{\;\;\,h-1} \right]_q
\end{equation}
since at least the $h+w-1$ unit squares which constitute the top row and the leftmost column are contained in such a polygon. So, by Proposition \ref{concpp} a Gaussian area limit law arises in this uniform fixed height and width ensemble, when both height and width tend to infinity.\\
In statistical physics \cite{PrOw} as well as in combinatorics \cite{Bou}, one is also interested in area limit laws in the so-called uniform fixed-perimeter ensembles, where all Ferrers diagrams of a fixed half-perimeter are considered equally likely. For fixed $m\ge 0$ the area generating function of  Ferrers diagrams of half-perimeter $m+2$ is obtained by summing formula \eqref{hwgf} over all pairs $(h,w)\in\mathbb{N}^2$ with $h+w=m+2.$ 
This can be written as
\begin{equation} \label{fpgf}
q^{m+1} \sum_{h=0}^{m}\left[\begin{array}{c}
m\\
h
\end{array}\right]_q.
\end{equation}
Observe that the index of summation in \eqref{fpgf} can be interpreted as the height minus one and that the total number of such Ferrers Diagrams is $2^m.$ So the following probability generating function (PGF)
\begin{equation} \label{fixedpergf}
2^{-m}q^{m+1}u \sum_{h=0}^{m}u^h\left[\begin{array}{c}
m\\
h
\end{array}\right]_q
\end{equation}
describes the ensemble of Ferrers diagrams of half-perimeter $m+2$ counted by height and area. The additional height parameter allows us to use the results for height- and width-ensembles dicussed above. More precisely, we condition the area variable in the perimeter-ensemble on the height variable. In what follows we will prove that the joint distribution of the (properly rescaled) height and area variables in this ensemble converge in distribution to the two-dimensional standard normal distribution. The following proposition from \cite{Fli} (see also \cite{Se}) is taylor made for this situation.
\begin{prop}\label{weakconv}
Let $X_m,Y_m$ be real valued random variables and let $Y_m$ be supported in a lattice $L_m:=\left\{\alpha_m+k\delta_m|k\in\mathbb{Z} \right\},$ where $\delta_m>0$ and $\alpha_m\in \mathbb{R},$ i.e. $\mathbb{P}(Y_m\in L_m)=1.$ Suppose $Y_m$ satisfies a local limit law $\mu$ with a density $g(y)$ w.r.t. the Lebesgue measure on $\mathbb{R}$ (this implies $\delta_m\to 0$), i.e. for all $y\in\mathbb{R}$ and every sequence $(y_m)$ with $y_m\in L_m$ and $y_m\to y$ we have $\mathbb{P}(Y_m=y_m)/\delta_m \to g(y).$  Suppose further that for $\mu$-almost all $y\in\mathbb{R}$ the conditional distributions $\mathbb{P}(X_m\in \cdot|Y_m=y_m)$ converge weakly to a measure $\nu(\cdot,y)$ ($y_m\to y,\,y_m \in L_m$). Then the joint distribution of $(X_m,Y_m)$ converges weakly to the measure $\nu$ defined by 
\[
\nu(A\times B):=\int_{B}\nu(A,y) {\rm d} \mu (y)
\]       
for all Borel sets $A,B\subseteq \mathbb{R}.$
\end{prop}
For $q=1$ formula \eqref{fixedpergf} is simply $2^{-m}u(1+u)^m,$ the PGF of the height random variable $H_m.$ It satisfies the assumptions of \cite[Theorem IX.14]{FlSe} and hence a Gaussian local limit law arises. The mean and standard deviation of the height are easily computed to be asymptotically equal to $m/2$ and $\sqrt{m}/2,$ respectively. So let $Y_m=2(H_m-m/2)/\sqrt{m}.$ The random variable $Y_m$ is supported in $\mathbb{Z}/\sqrt{m}.$ Let $y\in\mathbb{R}$ and $y_m\to y,$ where each $y_m\in\mathbb{Z}/\sqrt{m}.$ The area random variable $A_m,$ conditioned on the event $\left\{Y_m=y_m \right\},$ has PGF
\begin{equation}
q^{m+1}\binom{m}{m/2+y_m\sqrt{m}/2}^{-1}\left[\begin{array}{c}
m\\
m/2+y_m\sqrt{m}/2
\end{array}\right]_q.
\end{equation}
An application of Lemma \ref{expectation} and Lemma \ref{var} with $r=m/2+y_m\sqrt{m}/2, s=m-r$ and $t=1$ shows that the mean and the variance of the conditioned area variables are asymptotically equal to $m^2/8$ and $m^3/48,$ respectively. So according to Proposition \ref{concpp}, the rescaled area variable $X_m:=(A_m-m^2/8)/\sqrt{m^3/48}$ conditioned to $\left\{Y_m=y_m\right\}$ converges weakly to a Gaussian random variable with mean zero and variance one. Now Proposition \ref{weakconv} yields
\begin{prop}
Denote $A_m$ and $H_m$ the random variables of area and height of a Ferrers diagram in the uniform fixed perimeter ensemble described by \eqref{fixedpergf}. Let  
\[
X_m=\frac{A_m-m^2/8}{\sqrt{m^{3}/48}},\quad Y_m=\frac{H_m-m/2}{\sqrt{m}/2}.
\]
Then, as $m\to\infty,$ $\left(X_m,Y_m\right)$ converges weakly to the two-dimensional standard normal-distribution.   
\qed
\end{prop}
The last result can also be obtained by analysing a $q$-difference equation satisfied by the generating function of Ferrers diagrams counted by height, area and half-perimeter.

\section*{Acknowledgements}

The author thanks C. Richard for critical and helpful comments on the manuscript. He would also like to acknowledge financial support by the German Research
Council (DFG) within the CRC 701.


\begin{thebibliography}{99}

\bibitem{Bou}
M. Bousquet-M\'elou, A method for the enumeration of various classes of column-convex polygons, 
\textit{Discrete Math.} \textbf{154} (1996), 1-25. 

\bibitem{Br}
D. Bressoud, \textit{Proofs and Confirmations}, Cambridge: Cambridge University Press (1999).

\bibitem{Con}
G.M.~Constantine and T.H. Savits,
A multivariate Fa\`a di Bruno formula with applications,
{\it Trans.~Am.~Math.~Soc. \bf 348} (1996), 503--520.

\bibitem{Fe}
W. Feller, \textit{Probability Theory and its Applications} Vol.II, New York: Wiley (1971).

\bibitem{FlSe}
P. Flajolet and R. Sedgewick, \textit{Analytic Combinatorics}, Cambridge: Cambridge University Press (2008), to be published.

\bibitem{Fli}
M.A. Fligner, A note on limit theorems for joint distributions with applications to linear signed rank statistics, \textit{J. Roy. Statist. Soc. Ser. B} \textbf{43} (1981), 61--64.

\bibitem{PrOw}
T. Prellberg and A.L. Owczarek,
Stacking models of vesicles and compact clusters, 
\textit{J. Stat. Phys.} \textbf{80} (1995),  755--779. 

\bibitem{Se}
J. Sethuraman, Some limit theorems for joint distributions, \textit{Sankhy$\overline{\text{a}}$, Ser. A} \textbf{23} (1961), 379--386.

\bibitem{Wi}
D.B. Wilson, Diagonal sums of boxed plane partitions, \textit{Electron. J. Combin.} \textbf{8} (2001), Note 1.

\end{thebibliography}
\end{document}